\documentclass[oneside,english,review]{amsart}
\usepackage[T1]{fontenc}
\usepackage[latin9]{inputenc}
\usepackage{babel}

\usepackage{prettyref}
\usepackage{amsthm}
\usepackage{amssymb}
\usepackage{esint}
\usepackage[unicode=true, pdfusetitle,
 bookmarks=true,bookmarksnumbered=true,bookmarksopen=false,
 breaklinks=false,pdfborder={0 0 1},backref=page,colorlinks=false]
 {hyperref}

\makeatletter
\theoremstyle{plain}
\newtheorem{thm}{Theorem}
  \theoremstyle{definition}
  \newtheorem{defn}[thm]{Definition}
  \theoremstyle{plain}
  \newtheorem{prop}[thm]{Proposition}
  \theoremstyle{plain}
  \newtheorem{lem}[thm]{Lemma}
  \theoremstyle{remark}
  \newtheorem{rem}[thm]{Remark}
  \theoremstyle{plain}
  \newtheorem{cor}[thm]{Corollary}
 \theoremstyle{definition}
  \newtheorem{example}[thm]{Example}
  \theoremstyle{remark}
  \newtheorem*{acknowledgement*}{Acknowledgement}

\renewcommand{\Re}{\mathrm{Re}}
\renewcommand{\Im}{\mathrm{Im}}
\newrefformat{pro}{Proposition \ref{#1}}
\newrefformat{thm}{Theorem \ref{#1}}
\newrefformat{cla}{Claim \ref{#1}}
\newrefformat{lem}{Lemma \ref{#1}}
\newrefformat{cor}{Corollary \ref{#1}}
\newrefformat{rem}{Remark \ref{#1}}
\newrefformat{exa}{Example \ref{#1}}
\newrefformat{eq}{(\ref{#1})}
\numberwithin{equation}{section}
\numberwithin{thm}{section}

\makeatother

\begin{document}
\global\long\def\mc#1{\mathcal{#1}}

\global\long\def\mb#1{\mathbb{#1}}

\global\long\def\mr#1{\mathrm{#1}}

\title{Solution of a Loewner Chain Equation in Several Complex Variables}

\author{Mircea Voda}

\address{Department of Mathematics, University of Toronto, Toronto, ON M5S
2E4, Canada}

\email{mircea.voda@utoronto.ca}

\begin{abstract}
We find a solution to the Loewner chain equation in the case when
the infinitesimal generator satisfies $h\left(0,t\right)=0$, $Dh\left(0,t\right)=A$
for any $A\in L\left(\mb C^{n},\mb C^{n}\right)$ with $m\left(A\right)>0$.
We also study the related classes of spirallike mappings, mappings
with parametric representation and asymptotically spirallike mappings. 
\end{abstract}
\maketitle

\section{Introduction and preliminaries}

Subordination chains in several complex variables, the associated
differential equations and applications have been studied by various
authors (see \cite{MR0352510}, \cite{MR1892999}, \cite{MR2017933}, \cite{MR2438427}, \cite{MR2483641}, \cite{MR2507634}, \cite{arXiv10024262v1} 
and the references therein). Initially one assumed that the infinitesimal
generators of the subordination chains satisfied the normalization
$Dh\left(0,t\right)=I$ (and hence the chains satisfied $Df\left(0,t\right)=e^{t}I$).
Unlike the one variable situation it is not true that the non-normalized
case can be reduced to the $Dh\left(0,t\right)=I$ case (see \cite[p 413]{DurGraHamKoh10}).
Recently there has been interest in working with a more general normalization
(\cite{MR2438427},\cite{MR2483641}) or no normalization at all (\cite{MR2507634},\cite{arXiv10024262v1}).
In order to derive the existence of solutions to the Loewner chain
equation one needed to make further restrictions on $Dh\left(0,t\right)$
(\cite{MR2438427},\cite{MR2483641}) or to enlarge the {}``range''
of the Loewner chain (\cite{arXiv10024262v1}).

We will treat the situation when $Dh\left(0,t\right)=A$, where $A\in L\left(\mb C^{n},\mb C^{n}\right)$
is such that $m\left(A\right):=\min\left\{ \Re\left\langle A\left(z\right),z\right\rangle :\left\Vert z\right\Vert =1\right\} >0$.
More specifically, we are interested in studying the problems considered
by Graham, Hamada, Kohr and Kohr \cite{MR2438427} without using the
assumption that $k_{+}\left(A\right)<2m\left(A\right)$ (see \cite[Theorem 2.3]{MR2438427}
and \cite[Remark 2.8]{DurGraHamKoh10}). $k_{+}\left(A\right):=\max\left\{ \Re\lambda:\lambda\in\sigma\left(A\right)\right\} =\lim_{t\rightarrow\infty}\ln\left\Vert e^{tA}\right\Vert /t$
is the upper exponential (Lyapunov) index of $A$ ($\sigma\left(A\right)$
denotes the spectrum of $A$).
\begin{defn}
A mapping $f:B^{n}\times[0,\infty)\rightarrow\mathbb{C}^{n}$ is called
a subordination chain (Loewner chain) if $f\left(\cdot,t\right)$
is holomorphic (univalent) on $B^{n}$ and $f\left(z,s\right)=f\left(v\left(z,s,t\right),t\right)$,
$0\le s\le t$, where $v\left(\cdot,s,t\right)$ is a self-map of
$B^{n}$ fixing $0$ (in other words, $f\left(\cdot,s\right)$ is
subordinate to $f\left(\cdot,t\right)$) . $v$ is called the transition
mapping of the chain.
\end{defn}
Let \[
\mathcal{N}_{A}=\left\{ h\in H\left(B^{n}\right):\Re\left\langle h(z),z\right\rangle >0,\, Dh\left(0\right)=A\right\} \]
where $B^{n}$ denotes the unit ball in $\mathbb{C}^{n}$. Let $\mc H_{A}\left(B^{n}\right)$
be the class of mappings $h:B^{n}\times[0,\infty)\rightarrow\mathbb{C}^{n}$
such that $h\left(\cdot,t\right)\in\mathcal{N}_{A}$ for $t\ge0$
and $h\left(z,\cdot\right)$ is measurable on $[0,\infty)$ for $z\in B^{n}$.
Such mappings will be called infinitesimal generators. We study the
existence of solutions for the Loewner chain equation:\begin{equation}
\frac{\partial f}{\partial t}\left(z,t\right)=Df\left(z,t\right)h\left(z,t\right)\,\mathrm{a.e.}\, t\ge0,\, z\in B^{n}\label{eq:LE1}\end{equation}
where $h\in\mc H_{A}\left(B^{n}\right)$.

Throughout this paper we let $n_{0}:=\left[k_{+}\left(A\right)/m\left(A\right)\right]$.

By \cite[Theorem 2.1]{MR2438427} we know that the Loewner equation
for the transition mapping has a solution regardless of the value
of $n_{0}$. More precisely, we know that the initial value problem

\begin{equation}
\frac{\partial v}{\partial t}=-h\left(v,t\right)\,\mathrm{a.e.}\, t\ge s,\, v\left(z,s,s\right)=z,\, s\ge0\label{eq:TLE}\end{equation}
has a unique solution $v=v(z,s,t)$ such that $v\left(\cdot,s,t\right)$
is a univalent Schwarz mapping and $v\left(z,s,\cdot\right)$ is Lipschitz
continuous on $[s,\infty)$ locally uniformly with respect to $z$.
Furthermore we know that\begin{equation}
\frac{\left\Vert v\left(z,s,t\right)\right\Vert }{\left(1-\left\Vert v\left(z,s,t\right)\right\Vert \right)^{2}}\le e^{m\left(A\right)\left(s-t\right)}\frac{\left\Vert z\right\Vert }{\left(1-\left\Vert z\right\Vert \right)^{2}},\, z\in B^{n},\, t\ge s\ge0.\label{eq:transition inequality}\end{equation}

If $f\left(z,t\right)$ is a Loewner chain satisfying \prettyref{eq:LE1}
then $Df\left(0,t\right)=e^{tA}$ and we can write\begin{eqnarray*}
f\left(z,t\right) & = & e^{tA}\left(z+\sum_{k=2}^{\infty}F_{k}\left(z^{k},t\right)\right)\\
h\left(z,t\right) & = & Az+\sum_{k=2}^{\infty}H_{k}\left(z^{k},t\right)\end{eqnarray*}
where $F_{k}\left(\cdot,t\right)$ and $H_{k}\left(\cdot,t\right)$
are homogeneous polynomial mappings of degree $k$. We will denote
the Banach space of homogeneous polynomial mappings of degree $k$
from $\mathbb{C}^{n}$ to $\mathbb{C}^{n}$ by $\mathcal{P}^{k}\left(\mathbb{C}^{n}\right)$. 

Equating coefficients on both sides of \prettyref{eq:LE1} we get\begin{equation}
\frac{dF_{k}}{dt}(z^{k},t)=B_{k}\left(F_{k}\left(z^{k},t\right)\right)+N_{k}\left(z^{k},t\right),\,\mathrm{a.e.}\, t\in[0,\infty)\label{eq:coefficients}\end{equation}
where $B_{k}$ is a linear operator on $\mathcal{P}^{k}\left(\mathbb{C}^{n}\right)$
defined by \[
B_{k}\left(Q_{k}\left(z^{k}\right)\right)=kQ_{k}\left(Az,z^{k-1}\right)-AQ_{k}\left(z^{k}\right)\]
and $N_{k}\left(\cdot,t\right)\in\mathcal{P}^{k}\left(\mathbb{C}^{n}\right)$
is defined by\[
N_{k}\left(z^{k},t\right)=H_{k}\left(z^{k},t\right)+\sum_{j=2}^{k-1}jF_{j}\left(H_{k-j+1}\left(z^{k-j+1},t\right),z^{j-1},t\right).\]

We will say that a solution $f\left(z,t\right)$ of \prettyref{eq:LE1}
is polynomially bounded (bounded) if $\left\{ e^{-tA}f\left(\cdot,t\right)\right\} _{t\ge0}$
is locally polynomially bounded (locally bounded), i.e. for any compact
set $K\subset B^{n}$ there exists a constant $C_{K}$ and a polynomial
(constant polynomial) $P$ such that\[
\left\Vert e^{-tA}f\left(z,t\right)\right\Vert \le C_{K}P\left(t\right),\, z\in K,\, t\in[0,\infty).\]
The solutions of \prettyref{eq:coefficients} will be regarded as
functions $F_{k}:[0,\infty)\rightarrow\mc P^{k}\left(\mathbb{C}^{n}\right)$.
Consequently we will say that such $F_{k}$ are polynomially bounded
(bounded) if there exists a polynomial $ $(constant polynomial) $P$
such that $\left\Vert F_{k}\left(t\right)\right\Vert \le P\left(t\right)$,
$t\ge0$.

\prettyref{pro:normal structure} will show that a polynomially bounded
solution of \prettyref{eq:LE1} can be recovered from its first $n_{0}$
coefficients and the solution of \prettyref{eq:TLE}. Conversely,
\prettyref{thm:polybounded solutions} will show that by finding polynomially
bounded solutions to the first $n_{0}$ coefficient equations \prettyref{eq:coefficients}
we can find a solution of \prettyref{eq:LE1}. These results generalize
Poreda \cite[Theorem 4.1 and Theorem 4.4]{MR1104523}. Finally, after
a discussion about the existence of polynomially bounded solutions
to the coefficient equations we will obtain the main result, \prettyref{thm:existence of solutions},
that guarantees the existence of a Loewner chain solution for \prettyref{eq:LE1}.

We also consider what happens to the various classes of univalent
mappings that are related to Loewner chains. For convenience we recall
the definitions of the classes that we are considering, as given in
\cite{MR2438427}, where the case $n_{0}=1$ is treated. 

Let $\Omega\subset\mb C^{n}$ be a domain containing the origin.
\begin{defn}
We say that $\Omega$ is spirallike with respect to $A$ if $e^{-tA}w\in\Omega$
for any $w\in\Omega$ and $t\ge0$.
\end{defn}
\begin{defn}\label{def:asym spirallike}We say that $\Omega$ is
$A$-asymptotically spirallike if there exists a mapping $Q=Q\left(z,t\right):\Omega\times[0,\infty)\rightarrow\mb C^{n}$
that satisfies the following conditions:
\begin{enumerate}
\item $Q\left(\cdot,t\right)$ is a holomorphic mapping on $\Omega$, $Q\left(0,t\right)=0$,
$DQ\left(0,t\right)=A$, $t\ge0$, and the family $\left\{ Q\left(\cdot,t\right)\right\} _{t\ge0}$
is locally uniformly bounded on $\Omega$;
\item $Q\left(z,\cdot\right)$ is measurable on $[0,\infty)$ for all $z\in\Omega$;
\item the initial value problem\begin{equation}
\frac{\partial w}{\partial t}=-Q\left(w,t\right)\,\mathrm{a.e.}\, t\ge s,\, w\left(z,s,s\right)=z\label{eq:asym spir}\end{equation}
has a unique solution $w=w\left(z,s,t\right)$ for each $z\in\Omega$
and $s\ge0$, such that $w\left(\cdot,s,t\right)$ is a holomorphic
mapping of $\Omega$ into $\Omega$ for $t\ge s$, $w\left(z,s,\cdot\right)$
is locally absolutely continuous on $[s,\infty)$ locally uniformly
with respect to $z\in\Omega$ for $s\ge0$, and $\lim_{t\rightarrow\infty}e^{tA}w\left(z,0,t\right)=z$
locally uniformly on $\Omega$.
\end{enumerate}
\end{defn}

Let $f:B^{n}\rightarrow\mb C^{n}$ be a normalized univalent mapping,
i.e. such that $f\left(0\right)=0$ and $Df\left(0\right)=I$. $S\left(B^{n}\right)$
will denote the class of all such mappings. 
\begin{defn}
We say that $f$ is spirallike with respect to $A$ if $f\left(B^{n}\right)$
is spirallike. We will use $\hat{S}_{A}\left(B^{n}\right)$ to denote
the class of mappings that are spirallike with respect to $A$.
\end{defn}
\begin{defn}We say that $f$ is $A$-asymptotically spirallike if
$f\left(B^{n}\right)$ is $A$-asymptotically spirallike. $S_{A}^{a}\left(B^{n}\right)$
will denote the class of $A$-asymptotically spirallike mappings.
\end{defn}

\begin{defn}We say that $f$ has $A$-parametric representation if
there exists a mapping $h\in\mc H_{A}\left(B^{n}\right)$ such that
$f\left(z\right)=\lim_{t\rightarrow\infty}e^{tA}v\left(z,t\right)$
locally uniformly on $B^{n}$, where $v$ is the unique locally absolutely
continuous solution of the initial value problem\begin{equation}
\frac{\partial v}{\partial t}=-h\left(v,t\right)\, a.e.\, t\ge0,\, v\left(z,0\right)=z,\, z\in B^{n}.\label{eq:transition teq0}\end{equation}
$S_{A}^{0}\left(B^{n}\right)$ will denote the class of mappings with
$A$-parametric representation.\end{defn}

From \cite{MR2438427} we know that when $n_{0}=1$ we have that $S_{A}^{a}\left(B^{n}\right)=S_{A}^{0}\left(B^{n}\right)$
and the classes are compact. Also, since $\hat{S}_{A}\left(B^{n}\right)$
is a closed subset of $S_{A}^{a}\left(B^{n}\right)$ we also have
that $\hat{S}_{A}\left(B^{n}\right)$ is compact when $n_{0}=1$.
In \prettyref{sec:Spirallikeness,-parametric-representation,} we
will obtain a complete description of the $A$'s for which $\hat{S}_{A}\left(B^{n}\right)$
is compact (\prettyref{thm:compactness of spiralikeness}) and we
will study how the properties of $S_{A}^{0}\left(B^{n}\right)$ (\prettyref{exa:param rep not compact},
\prettyref{rem:parametric neq asymptotic}) and $S_{A}^{a}\left(B^{n}\right)$
(\prettyref{pro:normality for asymptotic}, \prettyref{rem:Loewner and asymptotic})
degenerate when $n_{0}>1$.

Throughout this paper the solutions of \prettyref{eq:LE1}, \prettyref{eq:TLE},
\prettyref{eq:coefficients} and \prettyref{eq:transition teq0} are
assumed to be locally absolutely continuous in $t$, locally uniformly
with respect to $z$.

\section{Solution of the Loewner chain equation}

We will repeatedly use the fact that given $\epsilon>0$ there exists
a constant $C_{\epsilon}$ such that \begin{equation}
\left\Vert e^{tA}\right\Vert \le C_{\epsilon}e^{t\left(k_{+}\left(A\right)+\epsilon\right)},\, t\ge0\label{eq:Banach exp bound}\end{equation}
(this follows immediately from the definition of $k_{+}\left(A\right)$).
In fact we can find a polynomial $P_{A}$ such that\begin{equation}
\left\Vert e^{tA}\right\Vert \le P_{A}\left(t\right)e^{tk_{+}\left(A\right)},\, t\ge0\label{eq:Cn exp bound}\end{equation}
 (see for example \cite[p 61, Exercise 16]{MR0352639}). Furthermore,
if $A$ is normal then\[
\left\Vert e^{tA}\right\Vert =e^{tk_{+}\left(A\right)},\, t\ge0.\]
Indeed, if we write $A=UDU^{*}$ where $D$ is a diagonal matrix and
$U$ is a unitary matrix then\[
\left\Vert e^{tA}\right\Vert =\left\Vert Ue^{tD}U^{*}\right\Vert =\left\Vert e^{tD}\right\Vert =e^{tk_{+}\left(D\right)}=e^{tk_{+}\left(A\right)}\]
(for non-Euclidean norms we just get $\left\Vert e^{tA}\right\Vert \le C_{A}e^{tk_{+}\left(A\right)}$). 

The following is a generalization of \cite[Theorem 4.1]{MR1104523}.
\begin{prop}
\label{pro:normal structure}If $f\left(z,t\right)$ is a polynomially
bounded solution of \prettyref{eq:LE1} such that\[
f\left(z,t\right)=e^{tA}\left(z+\sum_{k=2}^{\infty}F_{k}\left(z^{k},t\right)\right)\]
then \[
f\left(z,s\right)=\lim_{t\rightarrow\infty}e^{tA}\left(v\left(z,s,t\right)+\sum_{k=2}^{n_{0}}F_{k}\left(v\left(z,s,t\right)^{k},t\right)\right)\]
 and the limit is locally uniform in $z$.\end{prop}
\begin{proof}
\begin{multline*}
f\left(z,s\right)=f\left(v\left(z,s,t\right),t\right)\\
=e^{tA}\left(v\left(z,s,t\right)+\sum_{k=2}^{n_{0}}F_{k}\left(v\left(z,s,t\right)^{k},t\right)\right)+e^{tA}R\left(v\left(z,s,t\right),t\right)\end{multline*}
where $R\left(z,t\right)=\sum_{k=n_{0}+1}^{\infty}F_{k}\left(z^{k},t\right)$.
From the assumption on $f$, the formula for the remainder of the
Taylor series and Cauchy's formula, we easily get that $\left\{ R\left(\cdot,t\right)\right\} _{t\ge0}$
is locally polynomially bounded and in fact\[
\left\Vert R\left(z,t\right)\right\Vert \le C_{r}P\left(t\right)\left\Vert z\right\Vert ^{n_{0}+1},\,\left\Vert z\right\Vert \le r.\]
From the above and \prettyref{eq:transition inequality} we get

\begin{eqnarray*}
\left\Vert e^{tA}R\left(v\left(z,s,t\right),t\right)\right\Vert  & \le & C_{\epsilon}e^{t(k_{+}\left(A\right)+\epsilon)}P_{r}\left(t\right)\left\Vert v\left(z,s,t\right)\right\Vert ^{n_{0}+1}\\
 & \le & C_{\epsilon,r,s}e^{t\left(k_{+}\left(A\right)+\epsilon-\left(n_{0}+1\right)m\left(A\right)\right)}P\left(t\right),\,\left\Vert z\right\Vert \le r.\end{eqnarray*}
Hence, taking $\epsilon$ small enough we can conclude that $e^{tA}R\left(v\left(z,s,t\right),t\right)\rightarrow0$
locally uniformly.
\end{proof}
In order to prove \prettyref{thm:polybounded solutions} we will need
the following lemmas.
\begin{lem}
\label{lem:exp(A)exp(B)...}If $Q_{k}\in\mathcal{P}^{k}\left(\mathbb{C}^{n}\right)$
then the following identities hold for $t\in\mb R$\begin{eqnarray}
e^{tA}e^{tB_{k}}Q_{k}\left(\left(e^{-tA}z\right)^{k}\right) & = & Q_{k}\left(z^{k}\right)\label{eq:AB-A}\\
e^{tA}Q_{k}\left(\left(e^{-tA}z\right)^{k}\right) & = & e^{-tB_{k}}Q\left(z^{k}\right)\label{eq:A-A}\\
e^{tA}e^{tB_{k}}Q_{k}\left(z^{k}\right) & = & Q_{k}\left(\left(e^{tA}z\right)^{k}\right)\label{eq:AB}\end{eqnarray}
.\end{lem}
\begin{proof}
Define $A_{k}$ on $\mathcal{P}^{k}\left(\mathbb{C}^{n}\right)$ by
$A_{k}\left(Q_{k}\left(z^{k}\right)\right)=AQ_{k}\left(z^{k}\right)$.
One easily sees that $e^{tA_{k}}\left(Q_{k}\left(z^{k}\right)\right)=e^{tA}Q_{k}\left(z^{k}\right)$,
$A_{k}B_{k}=B_{k}A_{k}$ and \[
\left(A_{k}+B_{k}\right)\left(Q_{k}\left(z^{k}\right)\right)=kQ_{k}\left(Az,z^{k-1}\right).\]

For \prettyref{eq:AB-A} it is enough to check that $\phi\left(t\right)=e^{t\left(A_{k}+B_{k}\right)}Q_{k}\left(\left(e^{-tA}z\right)^{k}\right)$
satisfies $\phi'\left(t\right)=0$. Indeed \begin{eqnarray*}
\phi'\left(t\right) & = & e^{t\left(A_{k}+B_{k}\right)}\left[\left(A_{k}+B_{k}\right)\left(Q_{k}\left(\left(e^{-tA}z\right)^{k}\right)\right)-kQ_{k}\left(e^{-tA}Az,\left(e^{-tA}z\right)^{k-1}\right)\right]=0.\end{eqnarray*}

The last two identities follow immediately from the first one.\end{proof}
\begin{lem}
\label{lem:unique initial data}If $F_{k},G_{k}:[0,\infty)\rightarrow\mathcal{P}^{k}\left(\mathbb{C}^{n}\right)$,
$k=2,\ldots,m$ are solutions of \prettyref{eq:coefficients} and
the limits \begin{eqnarray*}
f\left(z,s\right) & := & \lim_{t\rightarrow\infty}e^{tA}\left(v\left(z,s,t\right)+\sum_{k=2}^{m}F_{k}\left(v\left(z,s,t\right)^{k},t\right)\right)\\
g\left(z,s\right) & := & \lim_{t\rightarrow\infty}e^{tA}\left(v\left(z,s,t\right)+\sum_{k=2}^{m}G_{k}\left(v\left(z,s,t\right)^{k},t\right)\right)\end{eqnarray*}
exist locally uniformly in $z\in B^{n}$ for some $s\ge0$, then $f\left(\cdot,s\right)=g\left(\cdot,s\right)$
if and only if $F_{k}=G_{k}$, $k=2,\ldots,m$.\end{lem}
\begin{proof}
Assume that $f\left(\cdot,s\right)=g\left(\cdot,s\right)$ and that
there exists $k\in\left\{ 2,\ldots,m\right\} $ such that $F_{k}\neq G_{k}$.
Let $k_{0}$ be the minimal such $k$. 

Equating the $k_{0}$-th coefficients of $f\left(\cdot,s\right)$
and $g\left(\cdot,s\right)$ and taking into account the minimality
of $k_{0}$ we get\begin{equation}
\lim_{t\rightarrow\infty}e^{tA}\left(F_{k_{0}}\left(\left(e^{\left(s-t\right)A}z\right)^{k_{0}},t\right)-G_{k_{0}}\left(\left(e^{\left(s-t\right)A}z\right)^{k_{0}},t\right)\right)=0.\label{eq:k0coeff}\end{equation}
Using \prettyref{eq:A-A} and the fact that $F_{k_{0}}$ is a solution
of \prettyref{eq:coefficients} we get\begin{eqnarray*}
e^{tA}F_{k_{0}}\left(\left(e^{\left(s-t\right)A}z\right)^{k_{0}},t\right) & = & e^{sA}e^{\left(s-t\right)B_{k}}F_{k_{0}}\left(z^{k_{0}},t\right)\\
 & = & e^{sA}e^{sB_{k}}\left(F_{k_{0}}\left(z^{k_{0}},0\right)+\int_{0}^{t}e^{-sB_{k_{0}}}N_{k_{0}}\left(z^{k_{0}},s\right)ds\right).\end{eqnarray*}
We can get an analogous identity for $G_{k_{0}}$ and then \prettyref{eq:k0coeff}
becomes\[
e^{sA}e^{sB_{k}}\left(F_{k_{0}}\left(z^{k_{0}},0\right)-G_{k_{0}}\left(z^{k_{0}},0\right)\right)=0.\]
Since $F_{k_{0}}$ and $G_{k_{0}}$ satisfy the same differential
equation with the same initial condition we have $F_{k_{0}}=G_{k_{0}}$,
thus reaching a contradiction.\end{proof}
\begin{lem}
\label{lem:growth estimate}If $P$ is a polynomial such that $P\left(t\right)\ge0$
for $t\ge s$ then \[
\int_{s}^{\infty}P\left(t\right)\left\Vert e^{\left(t-s\right)A}\right\Vert \frac{\left\Vert v\left(z,s,t\right)\right\Vert ^{n_{0}+1}}{\left(1-\left\Vert v\left(z,s,t\right)\right\Vert \right)^{2}}dt\le\frac{Q_{\epsilon,A,P}\left(s\right)}{\left(1-\left\Vert z\right\Vert \right)^{2\frac{k_{+}\left(A\right)}{m\left(A\right)}+\epsilon}},\,\epsilon>0\]
where $Q_{\epsilon,A,P}$ is a polynomial of the same degree as $P$.\end{lem}
\begin{proof}
Let\[
\alpha=\frac{k_{+}\left(A\right)}{m\left(A\right)}+\frac{\epsilon}{2}.\]
 We can restrict to the case when $\epsilon$ is small enough so that
$\alpha<n_{0}+1$. Using \prettyref{eq:transition inequality} we
see that\[
\frac{\left\Vert v\left(z,s,t\right)\right\Vert ^{n_{0}+1}}{\left(1-\left\Vert v\left(z,s,t\right)\right\Vert \right)^{2}}\le\frac{\left\Vert v\left(z,s,t\right)\right\Vert ^{\alpha}}{\left(1-\left\Vert v\left(z,s,t\right)\right\Vert \right)^{2}}\le\frac{e^{\left(s-t\right)\alpha m\left(A\right)}}{\left(1-\left\Vert z\right\Vert \right)^{2\alpha}}.\]

Let $\epsilon'$ be small enough so that\[
\left\Vert e^{\left(t-s\right)A}\right\Vert \le C_{\epsilon'}e^{\left(t-s\right)\left(k_{+}\left(A\right)+\epsilon'\right)}\]
and $\delta:=\alpha m\left(A\right)-k_{+}\left(A\right)-\epsilon'>0$.
Then \[
\int_{s}^{\infty}P\left(t\right)\left\Vert e^{\left(t-s\right)A}\right\Vert \frac{\left\Vert v\left(z,s,t\right)\right\Vert ^{n_{0}+1}}{\left(1-\left\Vert v\left(z,s,t\right)\right\Vert \right)^{2}}dt\le\frac{C_{\epsilon'}}{\left(1-\left\Vert z\right\Vert \right)^{2\alpha}}\int_{s}^{\infty}P\left(t\right)e^{-\delta(t-s)}dt\]
and it is not hard to see that \[
Q_{\epsilon,A,P}\left(s\right):=C_{\epsilon'}\int_{s}^{\infty}P\left(t\right)e^{-\delta\left(t-s\right)}dt\]
satisfies our requirements.\end{proof}
\begin{rem}
\label{rem:growth sharpness}When $A$ is normal, $P$ is constant
and $k_{+}\left(A\right)/m\left(A\right)>1$ we can sharpen the above
bound by letting $\epsilon=0$. 

Let $\beta=n_{0}+1-\alpha$, then using \prettyref{eq:transition inequality}
we get \[
\frac{\left\Vert v\left(z,s,t\right)\right\Vert ^{n_{0}+1}}{\left(1-\left\Vert v\left(z,s,t\right)\right\Vert \right)^{2}}\le\frac{e^{\left(s-t\right)\alpha m\left(A\right)}}{\left(1-\left\Vert z\right\Vert \right)^{2\alpha}}\left\Vert v\left(z,s,t\right)\right\Vert ^{\beta}\left(1-\left\Vert v\left(z,s,t\right)\right\Vert \right)^{2\left(\alpha-1\right)}.\]
If $A$ is normal we know that \[
\left\Vert e^{\left(t-s\right)A}\right\Vert =e^{\left(t-s\right)k_{+}\left(A\right)}\]
and hence we get \begin{multline*}
\int_{s}^{\infty}\left\Vert e^{\left(t-s\right)A}\right\Vert \frac{\left\Vert v\left(z,s,t\right)\right\Vert ^{n_{0}+1}}{\left(1-\left\Vert v\left(z,s,t\right)\right\Vert \right)^{2}}dt\\
\le\frac{1}{\left(1-\left\Vert z\right\Vert \right)^{2\alpha}}\int_{s}^{\infty}\left\Vert v\left(z,s,t\right)\right\Vert ^{\beta}\left(1-\left\Vert v\left(z,s,t\right)\right\Vert \right)^{2\left(\alpha-1\right)}dt.\end{multline*}

From the proof of \cite[Theorem 2.1]{MR2438427} we know that \[
-\frac{1+\left\Vert v\left(z,s,t\right)\right\Vert }{1-\left\Vert v\left(z,s,t\right)\right\Vert }\frac{1}{\left\Vert v\left(z,s,t\right)\right\Vert }\frac{d\left\Vert v\left(z,s,t\right)\right\Vert }{dt}\ge m\left(A\right).\]
Using the above inequality it is easy to conclude that\begin{multline*}
\int_{s}^{\infty}\left\Vert v\left(z,s,t\right)\right\Vert ^{\beta}\left(1-\left\Vert v\left(z,s,t\right)\right\Vert \right)^{2\left(\alpha-1\right)}dt\\
\le-\frac{2}{m\left(A\right)}\int_{s}^{\infty}\left\Vert v\left(z,s,t\right)\right\Vert ^{\beta-1}\left(1-\left\Vert v\left(z,s,t\right)\right\Vert \right)^{2\alpha-3}\frac{d\left\Vert v\left(z,s,t\right)\right\Vert }{dt}dt\\
=\frac{2}{m\left(A\right)}\int_{0}^{\left\Vert z\right\Vert }u^{\beta-1}\left(1-u\right)^{2\alpha-3}du\\
\le\frac{2}{m\left(A\right)}\int_{0}^{1}u^{\beta-1}\left(1-u\right)^{2\alpha-3}du.\end{multline*}
The last integral converges because by our assumptions $\beta-1>-1$
and $2\alpha-3>-1$. This completes the proof of our claim.\end{rem}
\begin{lem}
\label{lem:univalence}If $F_{k}:[0,\infty)\rightarrow\mathcal{P}^{k}\left(\mathbb{C}^{n}\right)$,
$k=2,\ldots,m$ are polynomially bounded and the limit \[
f\left(z,s\right)=\lim_{t\rightarrow\infty}e^{tA}\left(v\left(z,s,t\right)+\sum_{k=2}^{m}F_{k}\left(v\left(z,s,t\right)^{k},t\right)\right)\]
exists locally uniformly in $z\in B^{n}$ for some $s\ge0$ then $f\left(\cdot,s\right)$
is univalent.\end{lem}
\begin{proof}
First note that if $Q\in\mc P^{k}\left(\mb C^{n}\right)$ then\begin{eqnarray}
\left\Vert Q\left(z^{k}\right)-Q\left(w^{k}\right)\right\Vert  & = & \left\Vert \sum_{j=0}^{k-1}Q\left(z-w,z^{j},w^{k-1-j}\right)\right\Vert \nonumber \\
 & \le & \left\Vert Q\right\Vert \left\Vert z-w\right\Vert \sum_{j=0}^{k-1}\left\Vert z\right\Vert ^{j}\left\Vert w\right\Vert ^{k-1-j}.\label{eq:F(zk)-F(wk)}\end{eqnarray}
Using the above and \prettyref{eq:transition inequality} we see that\begin{multline*}
\left\Vert \sum_{k=2}^{m}F_{k}\left(v\left(z_{1},s,t\right)^{k},t\right)-\sum_{k=2}^{m}F_{k}\left(v\left(z_{2},s,t\right)^{k},t\right)\right\Vert \\
\le C_{r}P\left(t\right)e^{\left(s-t\right)m\left(A\right)}\left\Vert v\left(z_{1},s,t\right)-v\left(z_{2},s,t\right)\right\Vert ,\,\left\Vert z_{1}\right\Vert ,\left\Vert z_{2}\right\Vert \le r\end{multline*}
where $P$ is a polynomial bound on $F_{k}$, $k=2,\ldots,m$. For
sufficiently large $t$ we get\begin{multline*}
\left\Vert \sum_{k=2}^{m}F_{k}\left(v\left(z_{1},s,t\right)^{k},t\right)-\sum_{k=2}^{m}F_{k}\left(v\left(z_{2},s,t\right)^{k},t\right)\right\Vert \\
<\left\Vert v\left(z_{1},s,t\right)-v\left(z_{2},s,t\right)\right\Vert ,\,\left\Vert z_{1}\right\Vert ,\left\Vert z_{2}\right\Vert \le r\end{multline*}
 which implies that for sufficiently large $t$ \[
v\left(z,s,t\right)+\sum_{k=2}^{m}F_{k}\left(v\left(z,s,t\right)^{k},t\right)\]
is univalent on the ball $\left\Vert z\right\Vert \le r$. Now the
conclusion follows easily.
\end{proof}
The following consequence together with \prettyref{thm:polybounded solutions}
generalizes \cite[Theorem 2.6]{MR2438427}.
\begin{cor}
\label{cor:polybounded -> Loewner}All polynomially bounded solutions
of \prettyref{eq:LE1} are Loewner chains.\end{cor}
\begin{proof}
This follows from \prettyref{pro:normal structure} and \prettyref{lem:univalence}.\end{proof}
\begin{thm}
\label{thm:polybounded solutions}If $F_{k}$, $k=2,\ldots,n_{0}$
are polynomially bounded solutions of \prettyref{eq:coefficients}
then\[
g\left(z,s\right):=\lim_{t\rightarrow\infty}e^{tA}\left(v\left(z,s,t\right)+\sum_{k=2}^{n_{0}}F_{k}\left(v\left(z,s,t\right)^{k},t\right)\right)\]
exists locally uniformly with respect to $z$ and is a polynomially
bounded Loewner chain solution of \prettyref{eq:LE1}. If $F\left(t\right)$
is a polynomial bound for $F_{k}$, $k=2,\ldots,n_{0}$ then, given
$\epsilon>0$, there exists a polynomial $Q_{\epsilon,A,F}$ of the
same degree as $F$ such that\[
\left\Vert e^{-tA}g\left(z,t\right)\right\Vert \le\frac{Q_{\epsilon,A,F}\left(t\right)}{\left(1-\left\Vert z\right\Vert \right)^{2\frac{k_{+}\left(A\right)}{m\left(A\right)}+\epsilon}},\, z\in B^{n},t\ge0.\]
 Furthermore, if \[
g\left(z,t\right)=e^{tA}\left(z+\sum_{k=2}^{\infty}G_{k}\left(z^{k},t\right)\right)\]
then $G_{k}=F_{k}$, $k=2,\ldots,n_{0}$. \end{thm}
\begin{proof}
Let \[
u\left(z,s,t\right)=e^{tA}\left(v\left(z,s,t\right)+\sum_{k=2}^{n_{0}}F_{k}\left(v\left(z,s,t\right)^{k},t\right)\right).\]
We begin by showing that $\lim_{t\rightarrow\infty}u\left(z,s,t\right)$
exists locally uniformly.

It is easy to see that $u\left(z,s,t\right)$ is locally absolutely
continuous in $t$, so\begin{equation}
\left\Vert u\left(z,s,t_{1}\right)-u\left(z,s,t_{2}\right)\right\Vert =\left\Vert \int_{t_{1}}^{t_{2}}\frac{\partial u}{\partial t}\left(z,s,t\right)dt\right\Vert \le\int_{t_{1}}^{t_{2}}\left\Vert \frac{\partial u}{\partial t}\left(z,s,t\right)\right\Vert dt,\, s\le t_{1}\le t_{2}.\label{eq:cauchy estimate for u}\end{equation}
Now

\begin{multline*}
\frac{\partial u}{\partial t}\left(z,s,t\right)=e^{tA}A\left(v\left(z,s,t\right)+\sum_{k=2}^{n_{0}}F_{k}\left(v\left(z,s,t\right)^{k},t\right)\right)-e^{tA}\Biggl(h\left(v\left(z,s,t\right),t\right)+\\
\sum_{k=2}^{n_{0}}kF_{k}\left(v\left(z,s,t\right)^{k-1},h\left(v\left(z,s,t\right),t\right),t\right)-\sum_{k=2}^{n_{0}}\frac{dF_{k}}{dt}\left(v\left(z,s,t\right)^{k},t\right)\Biggr).\end{multline*}

Let\[
R\left(z,t\right)=h\left(z,t\right)-Az-\sum_{k=2}^{n_{0}}H_{k}\left(z^{k},t\right).\]
 Similarly to the proof of \cite[Theorem 4.4]{MR1104523}, a straightforward
computation using the assumption that $F_{k}$, $k=2,\ldots,n_{0}$
satisfy \prettyref{eq:coefficients}, leads to\begin{eqnarray}
\frac{\partial u}{\partial t}\left(z,s,t\right) & = & -e^{tA}\left(R\left(v\left(z,s,t\right),t\right)-\sum_{k=2}^{n_{0}}kF_{k}\left(v\left(z,s,t\right)^{k-1},R\left(v\left(z,s,t\right),t\right)\right)\right)\nonumber \\
 &  & -e^{tA}\left(\sum_{k=2}^{n_{0}}\sum_{l=n_{0}-k+2}^{n_{0}}kF_{k}\left(v\left(z,s,t\right)^{k-1},H_{l}\left(v\left(z,s,t\right)^{l},t\right),t\right)\right).\label{eq:du/dt}\end{eqnarray}
Using the ideas from the proof of \cite[Theorem 1.2]{MR1892999} (cf.
\cite[Lemma 1.2]{MR2438427}) we get \begin{equation}
\left\Vert R\left(z,t\right)\right\Vert \le C_{A}\frac{\left\Vert z\right\Vert ^{n_{0}+1}}{\left(1-\left\Vert z\right\Vert \right)^{2}}\label{eq:Rh estimate}\end{equation}
From \prettyref{eq:du/dt}, \prettyref{eq:Rh estimate} and \prettyref{eq:transition inequality}
we get\[
\left\Vert \frac{\partial u}{\partial t}\left(z,s,t\right)\right\Vert \le P\left(t\right)\left\Vert e^{tA}\right\Vert \frac{\left\Vert v\left(z,s,t\right)\right\Vert ^{n_{0}+1}}{\left(1-\left\Vert v\left(z,s,t\right)\right\Vert \right)^{2}}\]
where $P$ is a polynomial depending only on $F$ and $A$ (because
the bounds on $H_{k}$ can be chosen to depend only on $A$). Substituting
this estimate into \prettyref{eq:cauchy estimate for u} we get \[
\left\Vert u\left(z,s,t_{1}\right)-u\left(z,s,t_{2}\right)\right\Vert \le\left\Vert e^{sA}\right\Vert \int_{t_{1}}^{t_{2}}P\left(t\right)\left\Vert e^{\left(t-s\right)A}\right\Vert \frac{\left\Vert v\left(z,s,t\right)\right\Vert ^{n_{0}+1}}{\left(1-\left\Vert v\left(z,s,t\right)\right\Vert \right)^{2}}dt.\]
Using \prettyref{lem:growth estimate} we can now conclude that $\lim_{t\rightarrow\infty}u\left(z,s,t\right)$
exists uniformly on compact subsets. 

From the semigroup property for $v$ we immediately get that \begin{equation}
g\left(v\left(z,s,t\right),t\right)=g\left(z,s\right),\,0\le s\le t\label{eq:Loewner transition}\end{equation}
and by \prettyref{lem:univalence} we can conclude that $g\left(z,t\right)$
is a Loewner chain. Differentiating \prettyref{eq:Loewner transition}
with respect to $t$ and then letting $s\nearrow t$ we see that $g$
is a solution of \prettyref{eq:LE1}.

By the same considerations as above we get\[
\left\Vert e^{-sA}\left(u\left(z,s,t_{1}\right)-u\left(z,s,t_{2}\right)\right)\right\Vert \le\int_{t_{1}}^{t_{2}}P\left(t\right)\left\Vert e^{\left(t-s\right)A}\right\Vert \frac{\left\Vert v\left(z,s,t\right)\right\Vert ^{n_{0}+1}}{\left(1-\left\Vert v\left(z,s,t\right)\right\Vert \right)^{2}}dt.\]
Letting $t_{1}=s$ and $t_{2}\rightarrow\infty$ and using \prettyref{lem:growth estimate}
we get\begin{eqnarray}
\left\Vert e^{-sA}g\left(z,s\right)\right\Vert  & \le & \left\Vert z\right\Vert +\int_{s}^{\infty}P\left(t\right)\left\Vert e^{\left(t-s\right)A}\right\Vert \frac{\left\Vert v\left(z,s,t\right)\right\Vert ^{n_{0}+1}}{\left(1-\left\Vert v\left(z,s,t\right)\right\Vert \right)^{2}}dt\nonumber \\
 & \le & \frac{Q_{\epsilon,A,F}\left(s\right)}{\left(1-\left\Vert z\right\Vert \right)^{2\frac{k_{+}\left(A\right)}{m\left(A\right)}+\epsilon}},\,\epsilon>0\label{eq:Loewner growth}\end{eqnarray}
(P depends on $A$ and $F$). The above guarantees that the solution
is polynomially bounded.

The last statement follows from \prettyref{pro:normal structure}
and \prettyref{lem:unique initial data}.
\end{proof}
For the proof of \prettyref{thm:existence of solutions} we will need
some basic facts about ordinary differential equations. We will use
\cite{MR0352639} for this, but we are interested only in the finite
dimensional case. 

Let $X$ be a finite dimensional Banach space and $L$ be a bounded
linear operator on $X$. We consider the equation\begin{equation}
\frac{dx}{dt}=Lx+f\left(t\right),\,\mathrm{a.e.}\, t\ge0\label{eq:inhomo}\end{equation}
where $f:[0,\infty)\rightarrow X$ is a locally Lebesgue integrable
function. We know that any (locally absolutely continuous) solution
of \prettyref{eq:inhomo} is of the form\[
x\left(t\right)=e^{tL}x\left(0\right)+\int_{0}^{t}e^{\left(t-s\right)L}f\left(s\right)ds.\]
Note that the local Lebesgue integrability of $f$ is needed to ensure
the differentiability of the solution above, which follows from the
Lebesgue differentiation theorem (see \cite[Theorem 3.21]{MR1681462}).

We will use the following notations\begin{eqnarray*}
\sigma_{+}\left(L\right) & = & \left\{ \lambda\in\sigma\left(L\right):\Re\lambda>0\right\} \\
\sigma_{\le}\left(L\right) & = & \left\{ \lambda\in\sigma\left(L\right):\Re\lambda\le0\right\} \\
\sigma_{0}\left(L\right) & = & \left\{ \lambda\in\sigma\left(L\right):\Re\lambda=0\right\} .\end{eqnarray*}
 $P^{+}$, $P^{\le}$ and $P^{0}$ will denote the spectral projections
corresponding to $\sigma_{+}\left(L\right)$, $\sigma_{\le}\left(L\right)$
and $\sigma_{0}\left(L\right)$ respectively (see e.g. \cite[p 19]{MR0352639}
).

We say that $f$ is polynomially bounded if there exists a polynomial
$P$ such that $\left\Vert f\left(t\right)\right\Vert \le P\left(t\right)$,
$t\ge0$

Following the proof of \cite[Chapter II, Theorem 4.2]{MR0352639}
it is straightforward to check that if $f$ is polynomially bounded
then to each element $x_{0}^{\le}\in P^{\le}\left(X\right)$ there
corresponds a unique polynomially bounded solution of \prettyref{eq:inhomo}
that satisfies $P^{\le}x\left(0\right)=x_{0}^{\le}$. This solution
is given by the formula\begin{equation}
x\left(t\right)=e^{\left(t-t_{0}\right)L}x_{0}^{\le}+\int_{0}^{\infty}G_{L}\left(t-s\right)f\left(s\right)ds\label{eq:bounded solution}\end{equation}
where\begin{equation}
G_{L}\left(t\right)=\begin{cases}
e^{tL}P^{\le} & ,\, t\ge0\\
-e^{tL}P^{+} & ,\, t<0\end{cases}.\label{eq:polybounded green}\end{equation}
Furthermore if $f$ is bounded and $\sigma_{0}\left(L\right)=\emptyset$
then the solution \prettyref{eq:bounded solution} is bounded. Note
that in order to obtain a polynomial bound on the solution one needs
to use \prettyref{eq:Cn exp bound} rather than \prettyref{eq:Banach exp bound}.
\begin{rem}
\label{rem:infinite dimensional case}We are only interested in the
case when $X=\mathcal{P}^{k}\left(\mathbb{C}^{n}\right)$, but the
above considerations also apply to the case when $X$ is not finite
dimensional if we replace polynomially bounded by subexponential.
We say that $f$ is subexponential if for any $\epsilon>0$ there
exists $C_{\epsilon}$ such that $\left\Vert f\left(t\right)\right\Vert \le C_{\epsilon}e^{\epsilon t}$,
$t\ge0$. To be able to define the spectral projections we would also
need to require that $\sigma_{+}\left(L\right)$, $\sigma_{\le}\left(L\right)$
lie in different connected components of $\sigma\left(L\right)$.
\end{rem}
We will apply the above results to the coefficient equations \prettyref{eq:coefficients}
($X=\mathcal{P}^{k}\left(\mathbb{C}^{n}\right)$ and we regard the
coefficients as functions $F_{k}:[0,\infty)\rightarrow\mathcal{P}^{k}\left(\mathbb{C}^{n}\right)$),
hence we need information about the spectra of the operators $B_{k}$.
It is known (e.g. \cite[pp. 182-183]{MR947141}) that if $\lambda=\left(\lambda_{1},\ldots,\lambda_{n}\right)$
is a vector whose components are the (not necessarily distinct) eigenvalues
of $A$ then the eigenvalues of $B_{k}$ are \[
\left\{ \left\langle m,\lambda\right\rangle -\lambda_{s}:\,\left|m\right|=k,\, s\in\left\{ 1,\ldots,n\right\} \right\} \]
 where $m=\left(m_{1},\ldots,m_{n}\right)\in\mathbb{N}^{n}$ and $\left|m\right|=m_{1}+\ldots+m_{n}$.
Furthermore, if $A$ is a diagonal matrix then $B_{k}$ is also diagonal
and $z^{m}e_{s}$ is an eigenvector corresponding to $\left\langle m,\lambda\right\rangle -\lambda_{s}$
($e_{i}$, $i=1,\ldots,n$ denote the elements of the standard basis
of $\mathbb{C}^{n}$).

Following the terminology from \cite{MR947141} we will say that $A$
is nonresonant if $0\notin\sigma\left(B_{k}\right)$ for all $k$
(i.e. if the eigenvalues of $A$ are nonresonant; see \cite[p 180]{MR947141}).
Otherwise we say that $A$ is resonant. 
\begin{rem}
\label{rem:nonresonance}$0\notin\sigma\left(B_{k}\right)$ for all
$k>n_{0}$ (i.e. there are no resonances of order greater than $n_{0}$).
Indeed, if $m\in\mb N^{n}$, $\left|m\right|=k$ and $\lambda$ is
as above then\[
\Re\left(\left\langle m,\lambda\right\rangle -\lambda_{s}\right)\ge\left(n_{0}+1\right)k_{-}\left(A\right)-k_{+}\left(A\right)>0\]
where $k_{-}\left(A\right)=\min\left\{ \Re\lambda:\lambda\in\sigma\left(A\right)\right\} $.
For the last inequality we used the fact that $k_{-}\left(A\right)\ge m\left(A\right)$
and the definition of $n_{0}$ (which implies that $k_{+}\left(A\right)<\left(n_{0}+1\right)m\left(A\right)$).
In particular note that if $n_{0}=1$ then $A$ is nonresonant.

We will use $P_{k}^{+}$, $P_{k}^{\le}$, $P_{k}^{0}$ to denote the
projections associated with $B_{k}$. For $Q\in\mc P^{k}\left(\mb C^{n}\right)$
we will let $Q^{+}:=P_{k}^{+}Q$, $Q^{\le}:=P_{k}^{\le}Q$ and $Q^{0}:=P_{k}^{0}Q$. \end{rem}
\begin{thm}
\label{thm:existence of solutions}The equation \prettyref{eq:LE1}
always has a polynomially bounded Loewner chain solution that is uniquely
determined by the values of $F_{k}^{\le}\left(z^{k},0\right)$, $k=2,\ldots,n_{0}$,
which can be prescribed arbitrarily. Furthermore, if $A+\bar{A}$
is nonresonant then the solution can be chosen to be bounded.\end{thm}
\begin{proof}
This is an immediate consequence of \prettyref{thm:polybounded solutions}
and of the considerations on solutions of ordinary differential equations
from above. Note that $A+\bar{A}$ is nonresonant if and only if $\sigma_{0}\left(B_{k}\right)=\emptyset$,
$k\ge2$ .\end{proof}
\begin{rem}
It is not hard to see that if $A+\bar{A}$ is resonant, in general,
one cannot find a bounded solution (though it will be possible to
do this for particular choices of the infinitesimal generator $h$).
For example, one can choose $A$ such that $\sigma_{0}\left(B_{2}\right)=\left\{ 0\right\} $
and $0$ is a simple eigenvalue for $B_{2}$. In this case we would
have $e^{B_{2}}\vert_{P_{2}^{0}\left(\mc P^{2}\left(\mb C^{n}\right)\right)}=I_{P_{2}^{0}\left(\mc P^{2}\left(\mb C^{n}\right)\right)}$
and so\[
F_{2}^{0}\left(z^{2},t\right)=F_{2}^{0}\left(z^{2},0\right)+\int_{0}^{t}H_{2}^{0}\left(z^{2},s\right)ds.\]
In order to get a solution that is not bounded it is enough to choose
$h$ such that $\phi\left(t\right):=\int_{0}^{t}H_{2}^{0}\left(z^{2},s\right)ds$
is not bounded on $[0,\infty)$.\end{rem}
\begin{defn}
Let $\mc F\subset\prod_{k=2}^{n_{0}}P_{k}^{\le}\left(\mc P^{k}\left(\mb C^{n}\right)\right)$.
We define $S_{A}^{\mc F}\left(B^{n}\right)$ to be the family of mappings
$f\left(z\right)=z+\sum_{k=2}^{\infty}F_{k}\left(z^{k}\right)\in S\left(B^{n}\right)$
that can be embedded as the first element of a polynomially bounded
Loewner chain and such that $\left(F_{k}^{\le}\right)_{k=2,\ldots,n_{0}}\in\mc F$.

We want to study the compactness of the class $S_{A}^{\mc F}\left(B^{n}\right)$.
For this we need the following lemma that can be proved using similar
arguments to those in the proof of \cite[Lemma 2.8]{MR1982664} (cf.
\cite[Lemma 2.14]{MR2438427})\end{defn}
\begin{lem}
\label{lem:Loewner subsequence}Every sequence of Loewner chains $\left\{ f_{k}\left(z,t\right)\right\} $
such that $Df_{k}\left(0,t\right)=e^{tA}$ and \[
\left\Vert e^{-tA}f_{k}\left(z,t\right)\right\Vert \le C_{r}P\left(t\right),\,\left\Vert z\right\Vert \le r<1,\, t\ge0,\]
where $P\left(t\right)$ is a polynomial, has a subsequence that converges
locally uniformly on $B^{n}$ to a polynomially bounded Loewner chain
$f\left(z,t\right)$ for $t\ge0$.\end{lem}
\begin{thm}
\label{thm:compact class}If $\mc F\subset\prod_{k=2}^{n_{0}}P_{k}^{\le}\left(\mc P^{k}\left(\mb C^{n}\right)\right)$
is bounded (compact) then $S_{A}^{\mc F}\left(B^{n}\right)$ is normal
(compact). Furthermore, given $\epsilon>0$ there exists a constant
$C_{\epsilon,A,\mc F}$ such that\[
\left\Vert f\left(z\right)\right\Vert \le\frac{C_{\epsilon,A,\mc F}}{\left(1-\left\Vert z\right\Vert \right)^{2\frac{k_{+}\left(A\right)}{m\left(A\right)}+\epsilon}},\, f\in S_{A}^{\mc F}\left(B^{n}\right).\]
\end{thm}
\begin{proof}
Let $f\in S_{A}^{\mc F}\left(B^{n}\right)$ and $f\left(z,t\right)$
be a polynomially bounded Loewner chain such that $f\left(z,0\right)=f\left(z\right)$.
Suppose that \[
f\left(z,t\right)=e^{tA}\left(z+\sum_{k=2}^{\infty}F_{k}\left(z^{k},t\right)\right).\]
We know that (see \prettyref{eq:coefficients} and \prettyref{eq:bounded solution})
\[
F_{k}\left(z^{k},t\right)=e^{tB_{k}}F_{k}^{\le}\left(z^{k},0\right)+\int_{0}^{\infty}G_{B_{k}}\left(t-s\right)N_{k}\left(s\right)ds.\]
 Now it is straightforward to check that if $\mc F$ is bounded then
$F_{k}$, $k=2,\ldots,n_{0}$ can be bounded by a polynomial $F$
that doesn't depend on $f$ (it depends only on $\mc F$ and $A$).
By \prettyref{thm:polybounded solutions} we have\[
\left\Vert e^{-tA}f\left(z,t\right)\right\Vert \le\frac{Q_{\epsilon,A,F}\left(t\right)}{\left(1-\left\Vert z\right\Vert \right)^{2\frac{k_{+}\left(A\right)}{m\left(A\right)}+\epsilon}}.\]
When $t=0$ the above inequality proves the fact that $S_{A}^{\mc F}\left(B^{n}\right)$
is normal. Furthermore, if $\mc F$ is also closed we can now argue
by contradiction using the previous Lemma to see that $S_{A}^{\mc F}\left(B^{n}\right)$
is also closed.\end{proof}
\begin{rem}
It is not hard to see that the results of this section (except for
\prettyref{rem:growth sharpness}) remain true for any norm on $\mb C^{n}$.
Furthermore, with appropriate modifications (see \prettyref{rem:infinite dimensional case}
and \cite{MR2098737}) the results can be extended to reflexive complex
Banach spaces.
\end{rem}

\section{\label{sec:Spirallikeness,-parametric-representation,}Spirallikeness,
parametric representation, asymptotical spirallikeness}

We start by answering \cite[Open Problem 6.4.13]{MR2017933}. 
\begin{thm}
\label{thm:compactness of spiralikeness}$\hat{S}_{A}\left(B^{n}\right)$
is compact if and only if $A$ is nonresonant.\end{thm}
\begin{proof}
If $f\in\hat{S}_{A}\left(B^{n}\right)$ we know that \begin{equation}
f\left(z,t\right):=e^{tA}f\left(z\right)=e^{tA}\left(z+\sum_{k=2}^{\infty}F_{k}\left(z^{k}\right)\right)\label{eq:spirallike Loewner chain}\end{equation}
 is a Loewner chain (this follows easily from the definitions). It
is clear that $\hat{S}_{A}\left(B^{n}\right)\subset S_{A}^{\mc F}\left(B^{n}\right)$,
where \[
\mc F:=\left\{ \left(F_{k}^{\le}\right)_{k=2,\ldots,n_{0}}:f\left(z\right)=z+\sum_{k=2}^{\infty}F_{k}\left(z^{k}\right)\in\hat{S}_{A}\left(B^{n}\right)\right\} .\]
 It is easy to see that $\hat{S}_{A}\left(B^{n}\right)$ is closed
by using the analytic characterization \prettyref{eq:an. ch. of spir.}
and the fact that $\mc N_{A}$ is compact. Now, by \prettyref{thm:compact class},
if the coefficients $F_{k}$, $k=2,\ldots,n_{0}$ can be bounded independently
of $f$ then $\hat{S}_{A}\left(B^{n}\right)$ is compact. For our
particular Loewner chain \prettyref{eq:spirallike Loewner chain}
the coefficient equations \prettyref{eq:coefficients} take the simple
form $0=B_{k}F_{k}+N_{k}$. 

If $A$ is nonresonant then the operators $B_{k}$ are invertible
and hence $F_{k}=-B_{k}^{-1}N_{k}$. Now it is straightforward to
see that we can choose bounds for $F_{k}$, $k=2,\ldots,n_{0}$ that
don't depend on $f$, thus yielding compactness of $\hat{S}_{A}\left(B^{n}\right)$.

If $A$ is resonant then let $k_{0}\le n_{0}$ be the largest $k$
such that $B_{k}$ is singular (by \prettyref{rem:nonresonance} $B_{k}$
is not singular for $k>n_{0}$). Let $h\left(z\right)=Az+H_{k_{0}}\left(z^{k_{0}}\right)\in\mathcal{N}_{A}$,
where $H_{k_{0}}$ is chosen such that $B_{k_{0}}F_{k_{0}}+H_{k_{0}}=0$
has a solution. Note that for our particular $h$ we have $N_{k}=0$,
$k=2,\ldots,k_{0}-1$ and $N_{k_{0}}=H_{k_{0}}$. Since $B_{k}$,
$k>k_{0}$ are nonsingular there is no problem in solving for $F_{k}$,
$k>k_{0}$ and then, using \prettyref{thm:polybounded solutions},
we get that \[
f\left(z,s\right)=\lim_{t\rightarrow\infty}e^{tA}\left(v\left(z,s,t\right)+\sum_{k=2}^{n_{0}}F_{k}\left(v\left(z,s,t\right)^{k}\right)\right)\]
is a Loewner chain solution of \prettyref{eq:LE1} with $h\left(z,t\right)=h\left(z\right)$.
Since $h$ doesn't depend on $t$ we have $v\left(z,s,t\right)=v\left(z,0,t-s\right)$
and this yields that $f\left(z,s\right)=e^{sA}f\left(z,0\right)$.
Hence $f\left(\cdot,0\right)\in\hat{S}_{A}\left(B^{n}\right)$ and
by \prettyref{thm:polybounded solutions} it's $k_{0}$-th coefficient
is $F_{k_{0}}$. 

This construction works with any $F_{k_{0}}$ that is a solution of
$B_{k_{0}}F_{k_{0}}+H_{k_{0}}=0$. Since $B_{k_{0}}$ is singular,
the solutions of the equation form a non-trivial affine subspace of
$\mathcal{P}^{k_{0}}\left(\mathbb{C}^{n}\right)$, so in particular
there exist solutions of arbitrarily large norm. Now we can conclude
that there exist spirallike mappings with arbitrarily large $k_{0}$-th
coefficient. This proves that $\hat{S}_{A}\left(B^{n}\right)$ is
not compact when $A$ is resonant.\end{proof}
\begin{rem}
\label{rem: solution of spir. eq.}Let $h\in\mathcal{N}_{A}$. By
the same ideas as in the proof of the previous theorem we can conclude
that if $A$ is nonresonant then the equation\begin{equation}
Df\left(z\right)h\left(z\right)=Af\left(z\right)\label{eq:an. ch. of spir.}\end{equation}
has a unique holomorphic solution, which is in fact biholomorphic
(because of \prettyref{cor:polybounded -> Loewner}). By \prettyref{rem:nonresonance}
this generalizes \cite[Corollary 4.8]{DurGraHamKoh10}. On the other
hand, if $A$ is resonant, there either is no holomorphic solution
(for example if $H_{2}\notin B_{2}\left(\mc P^{2}\left(\mb C^{n}\right)\right)$)
or the holomorphic solutions (in fact, biholomorphic) are not unique. 
\end{rem}

\begin{rem}\label{rem:spirallike growth}As a consequence of the
proof of \prettyref{thm:compactness of spiralikeness} and of \prettyref{thm:compact class}
we have the following bound for mappings in $\hat{S}_{A}\left(B^{n}\right)$
:\[
\left\Vert f\left(z\right)\right\Vert \le\frac{C_{\epsilon,A}}{\left(1-\left\Vert z\right\Vert \right)^{2\frac{k_{+}\left(A\right)}{m\left(A\right)}+\epsilon}},\, z\in B^{n},\,\epsilon>0,\, f\in\hat{S}_{A}\left(B^{n}\right)\]
(cf. \cite[Theorem 3.1]{MR1845250} and \cite[Theorem 12]{MR000000}).
Furthermore, if $A$ is normal the above estimate holds with $\epsilon=0$
(the case $k_{+}\left(A\right)/m\left(A\right)>1$ follows using \prettyref{rem:growth sharpness},
while the case $k_{+}\left(A\right)/m\left(A\right)=1$ is covered
by \cite[Corollary 3.1]{MR1845250})\end{rem}
\begin{rem}
Let $A=\mathrm{diag}\left(\lambda_{1},\ldots,\lambda_{n}\right)$
($\Re\lambda_{i}>0$) and $m\in\mathbb{N}^{n}$ with $m_{i}=0$, $i=1,\ldots,s$,
where $1\le s<n$. Then it is easy to compute that for $f\left(z\right)=z+az^{m}e_{s}$
we have\[
h\left(z\right)=\left[Df\left(z\right)\right]^{-1}Af\left(z\right)=Az+a\left(\lambda_{s}-\left\langle m,\lambda\right\rangle \right)z^{m}e_{s}.\]
If $\lambda_{s}-\left\langle m,\lambda\right\rangle =0$ we get that
$f\in\hat{S}_{A}\left(B^{n}\right)$ for any $a\in\mathbb{C}^{n}$
generalizing an example from \cite[p 57]{MR1845250}. If $\lambda_{s}-\left\langle m,\lambda\right\rangle \neq0$
then $f\in\hat{S}_{A}\left(B^{n}\right)$ for any $a$ such that \[
\left|a\right|\le\frac{m\left(A\right)}{\left|\lambda_{s}-\left\langle m,\lambda\right\rangle \right|}.\]
This example suggests that in the case when $A$ is nonresonant a
sharp upper growth bound on $\hat{S}_{A}\left(B^{n}\right)$ would
have to depend on the entire spectrum of $A$. 
\end{rem}
Next we extend \cite[Corollary 2.2]{MR1897032}. For simplicity we
only treat the 2-dimensional case. $S^{*}\left(B^{n}\right)=\hat{S}_{I}\left(B^{n}\right)$
denotes the class of normalized starlike mappings. 
\begin{prop}
Let $A=\mathrm{diag\left(1,\lambda\right)}$, $\Re\lambda\ge1$. Define
$\Phi_{\alpha,\beta}:S^{*}\left(B^{1}\right)\rightarrow\hat{S}_{A}\left(B^{2}\right)$
by\[
\Phi_{\alpha,\beta}\left(f\right)\left(z\right)=\left(f\left(z_{1}\right),\left(\frac{f\left(z_{1}\right)}{z_{1}}\right)^{\alpha}\left(f'\left(z_{1}\right)\right)^{\beta}z_{2}\right).\]
If $\alpha\in[0,\Re\lambda]$ and $\beta\in\left[0,1/2\right]$ such
that $ $ $\alpha+\beta\le\Re\lambda$ then $\Phi_{\alpha,\beta}\left(S^{*}\left(B^{1}\right)\right)\subset\hat{S}_{A}\left(B^{2}\right)$.\end{prop}
\begin{proof}
We follow the proof of \cite[Theorem 2.1]{MR1897032}. Let $f\in S^{*}\left(B^{1}\right)$
and define \[
F\left(z,t\right)=e^{tA}\Phi_{\alpha,\beta}\left(f\right)\left(z\right)=\left(e^{t}f\left(z_{1}\right),e^{\lambda t}\left(\frac{f\left(z_{1}\right)}{z_{1}}\right)^{\alpha}\left(f'\left(z_{1}\right)\right)^{\beta}z_{2}\right).\]
 It is sufficient to check that $F\left(z,t\right)$ is a Loewner
chain. Because of the particular form of $F$ and by \prettyref{cor:polybounded -> Loewner}
it is enough to check that $F$ satisfies a Loewner chain equation,
i.e. that \[
h\left(\cdot,t\right):=\left[DF\left(\cdot,t\right)\right]^{-1}\frac{\partial F}{\partial t}\left(\cdot,t\right)\in\mc H_{A}\left(B^{2}\right),\,\mathrm{a.e.}\, t\ge0.\]

Let $p\left(z_{1}\right)=f\left(z_{1}\right)/\left(z_{1}f'\left(z_{1}\right)\right)$.
Straightforward computations yield that \[
h\left(z,t\right)=\left(z_{1}p\left(z_{1}\right),z_{2}\left(\lambda-\alpha-\beta+\left(\alpha+\beta\right)p\left(z_{1}\right)+\beta z_{1}p'\left(z_{1}\right)\right)\right).\]

The same arguments as in the proof of \cite[Theorem 2.1]{MR1897032}
(we are using the fact that $f\in S^{*}\left(B^{1}\right)$ implies
that $\Re p>0$) show that it is sufficient to check that \[
q\left(x\right)=(\Re\lambda-\alpha-\beta)x^{2}-2\beta x+\alpha+\beta\]
is non-negative on $\left[0,1\right]$. This follows by elementary
analysis.\end{proof}
\begin{rem}
Let $A$ be as in the Proposition above. For $\alpha=\Re\lambda-1/2$,
$\beta=1/2$ and $f\left(z\right)=z/\left(1-z\right)^{2}$ we can
see that $\Phi_{\alpha,\beta}\left(f\right)\in\hat{S}_{A}\left(B^{2}\right)$
attains the asymptotic growth bound from \prettyref{rem:spirallike growth}. 
\end{rem}
Next we consider the class of mappings with $A$-parametric representation.
Unlike the class of spirallike mappings, the class $S_{A}^{0}\left(B^{n}\right)$
is not compact when $n_{0}>1$, as we can see from the following example.
\begin{example}
\label{exa:param rep not compact}Let $A=\mathrm{diag}\left(\lambda,1\right)$,
$\Re\lambda\ge2$ and define \[
h\left(z,t\right)=\left(\lambda z_{1}+a\left(t\right)z_{2}^{2},z_{2}\right),\, z=\left(z_{1},z_{2}\right)\in B^{2}.\]
If for example $\left|a\left(t\right)\right|\le1$, $t\ge0$ it is
easy to check that $h\left(\cdot,t\right)\in\mc N_{A}$, $t\ge0$.
Then\[
v\left(z,t\right)=\left(e^{-\lambda t}\left(z_{1}-\left(\int_{0}^{t}a\left(s\right)e^{\left(\lambda-2\right)s}ds\right)z_{2}^{2}\right),e^{-t}z_{2}\right)\]
is the solution of \prettyref{eq:transition teq0}. When $\lim_{t\rightarrow\infty}e^{tA}v\left(z,t\right)$
exists locally uniformly on $B^{n}$ we get that $f\left(z\right)=\left(z_{1}-\left(\int_{0}^{\infty}a\left(s\right)e^{\left(\lambda-2\right)s}ds\right)z_{2}^{2},z_{2}\right)\in S_{A}^{0}\left(B^{2}\right)$.
Since the second coefficient of the Taylor series expansion can be
made arbitrarily large by an appropriate choice of $a\left(\cdot\right)$
we conclude that $S_{A}^{0}\left(B^{2}\right)$ is not compact. 

This example can be generalized for any $A$ by considering $h\left(z,t\right)=Az+a\left(t\right)H_{2}\left(z^{2}\right)\in\mc H_{A}\left(B^{n}\right)$
such that $H_{2}^{\le}\neq0$.
\end{example}
Next we consider the class $S_{A}^{a}\left(B^{n}\right)$. The following
characterization of $A$-asymptotically spirallike mappings is derived
from the proofs of \cite[Theorem 3.1 and Theorem 3.5]{MR2438427}.
\begin{prop}
\label{pro:asymptotic spirallikeness}Let $f:B^{n}\rightarrow\mb C^{n}$
be a holomorphic mapping and \[
f\left(z\right)=z+\sum_{k=2}^{\infty}F_{k}\left(z^{k}\right).\]
Then $f$ is $A$-asymptotically spirallike if and only if there exists
$h\in\mc H_{A}\left(B^{n}\right)$ such that \begin{equation}
f\left(z\right)=\lim_{t\rightarrow\infty}e^{tA}\left(v\left(z,t\right)+\sum_{k=2}^{n_{0}}F_{k}\left(v\left(z,t\right)^{k}\right)\right)\label{eq:as-spir-characterization}\end{equation}
locally uniformly on $B^{n}$, where $v$ is the solution of \prettyref{eq:transition teq0}.\end{prop}
\begin{proof}
First assume that $f$ is $A$-asymptotically spirallike. Hence there
exists a mapping $Q:f\left(B^{n}\right)\times[0,\infty)\rightarrow\mb C^{n}$
satisfying the assumptions from Definition \prettyref{def:asym spirallike}.
Let $\nu$ be the solution of the initial value problem \prettyref{eq:asym spir}.
By definition it will satisfy \[
\lim_{t\rightarrow\infty}e^{tA}\nu\left(f\left(z\right),0,t\right)=f\left(z\right)\]
locally uniformly on $B^{n}$.

Let $v$ be defined by $v\left(z,s,t\right)=f^{-1}\left(\nu\left(f\left(z\right),s,t\right)\right)$,
$z\in B^{n}$, $t\ge s$. Also, let $h\left(z,t\right)=\left[Df\left(z\right)\right]^{-1}Q\left(f\left(z\right),t\right)$,
$z\in B^{n}$, $t\ge0$. With the same proof as in \cite[Theorem 3.5]{MR2438427}
one sees that $h\in\mc H_{A}\left(B^{n}\right)$ and that $v$ is
the solution of \prettyref{eq:TLE}.

We have\[
f\left(z\right)=\lim_{t\rightarrow\infty}e^{tA}\nu\left(f\left(z\right),0,t\right)=\lim_{t\rightarrow\infty}e^{tA}f\left(v\left(z,0,t\right)\right)\]
locally uniformly on $B^{n}$. Like in the proof of \prettyref{pro:normal structure}
we also see that \[
\lim_{t\rightarrow\infty}e^{tA}f\left(v\left(z,0,t\right)\right)=\lim_{t\rightarrow\infty}e^{tA}\left(v\left(z,0,t\right)+\sum_{k=2}^{n_{0}}F_{k}\left(v\left(z,0,t\right)^{k}\right)\right)\]
yielding the desired conclusion (the fact that $f$ is univalent follows
from \prettyref{lem:univalence}).

Now assume that \prettyref{eq:as-spir-characterization} holds. The
conclusion follows exactly as in the proof of \cite[Theorem 3.1]{MR2438427}.\end{proof}
\begin{rem}
\label{rem:parametric neq asymptotic}From the above characterization
of $S_{A}^{a}\left(B^{n}\right)$ it is easy to see that $S_{A}^{a}\left(B^{n}\right)\neq S_{A}^{0}\left(B^{n}\right)$
when $n_{0}>1$.

In \prettyref{pro:normality for asymptotic} we obtain a partial result
about the normality of the class $S_{A}^{a}\left(B^{n}\right)$, but
first we need the following lemmas.\end{rem}
\begin{lem}
\label{lem:v component bound}Let $A=\mr{diag}\left(\lambda_{1},\ldots,\lambda_{n}\right)$
and $h\in\mc H_{A}\left(B^{n}\right)$. If $v=\left(v_{1},\ldots,v_{n}\right)$
is the solution of \prettyref{eq:transition teq0} then\[
\left\Vert v_{i}\left(z,t\right)\right\Vert \le C\begin{cases}
e^{-\Re\lambda_{i}t} & ,\,\Re\lambda_{i}<2m\left(A\right)\\
(1+t)e^{-\Re\lambda_{i}t} & ,\,\Re\lambda_{i}=2m\left(A\right)\\
e^{-2m\left(A\right)t} & ,\,\Re\lambda_{i}>2m\left(A\right)\end{cases}\]
where $C$ is a constant that depends on $A$, $\lambda_{i}$ and
$\left\Vert z\right\Vert $.\end{lem}
\begin{proof}
Writing $h=\left(h_{1},\ldots,h_{n}\right)$ and $\tilde{h}_{i}=h_{i}-\lambda_{i}v_{i}$,
\prettyref{eq:transition teq0} yields\[
\frac{dv_{i}}{dt}=-\lambda_{i}v_{i}-\tilde{h}_{i}\left(v,t\right).\]
Integrating we get \[
e^{t\lambda_{i}}v_{i}=z_{i}-\int_{0}^{t}e^{s\lambda_{i}}\tilde{h}_{i}\left(v,s\right)ds.\]
Hence\begin{eqnarray*}
\left\Vert e^{t\lambda_{i}}v_{i}\left(z,t\right)\right\Vert  & \le & \left|z_{i}\right|+\int_{0}^{t}e^{s\Re\lambda_{i}}\left\Vert h\left(v\left(z,s\right),s\right)-Av\left(z,s\right)\right\Vert ds\\
 & \le & \left|z_{i}\right|+C_{A,\left\Vert z\right\Vert }\int_{0}^{t}e^{s\Re\lambda_{i}}\left\Vert v\left(z,s\right)\right\Vert ^{2}ds\\
 & \le & \left|z_{i}\right|+C_{A,\left\Vert z\right\Vert }\int_{0}^{t}e^{s\left(\Re\lambda_{i}-2m\left(A\right)\right)}ds\\
 & \le & \begin{cases}
C_{A,\left\Vert z\right\Vert ,\lambda_{i}} & ,\,\Re\lambda_{i}<2m\left(A\right)\\
C_{A,\left\Vert z\right\Vert }(1+t) & ,\,\Re\lambda_{i}=2m\left(A\right)\\
C_{A,\left\Vert z\right\Vert ,\lambda_{i}}e^{\left(\Re\lambda_{i}-2m\left(A\right)\right)t} & ,\,\Re\lambda_{i}>2m\left(A\right)\end{cases}\end{eqnarray*}
(the second inequality follows from \prettyref{eq:Rh estimate} and
the third estimate follows from \prettyref{eq:transition inequality}).\end{proof}
\begin{lem}
\label{lem:bounded ass coefficients}Let $\lambda\in\mb C$ such that
$\Re\lambda\ge0$, $a\in\mb C$ and $h:[0,\infty)\rightarrow\mb C$
such that $\left|h\left(t\right)\right|\le C$, $t\ge0$. If \begin{equation}
\lim_{t\rightarrow\infty}\int_{0}^{t}e^{s\lambda}\left(h\left(s\right)+a\right)ds=0\label{eq:hyp int eq 0}\end{equation}
then $\left|a\right|\le C$.\end{lem}
\begin{proof}
We argue by contradiction. Assume that $\left|a\right|>C$. Then\[
\Re\left(\left(h\left(s\right)+a\right)\bar{a}\right)\ge\left|a\right|^{2}-\left|a\right|\left|h\left(s\right)\right|\ge\left|a\right|\left(\left|a\right|-C\right)=:\delta>0.\]

If $\Im\lambda=0$ we get \[
\Re\left(\left(\int_{0}^{t}e^{s\lambda}\left(h\left(s\right)+a\right)ds\right)\bar{a}\right)\ge t\delta\]
contradicting \prettyref{eq:hyp int eq 0}.

If $\Im\lambda\neq0$ we can find $\tau>0$ such that\[
\Re\left(e^{s\lambda}\left(h\left(s\right)+a\right)\bar{a}\right)\ge\frac{\delta}{2},\, s\in\left[\frac{2k\pi}{\Im\lambda},\frac{2k\pi}{\Im\lambda}+\tau\right],\, k\ge0,\, k\in\mb Z.\]
From \prettyref{eq:hyp int eq 0} we get that\[
\lim_{t\rightarrow\infty}\int_{t}^{t+\tau}e^{s\lambda}\left(h\left(s\right)+a\right)ds=0.\]
This is contradicted by\[
\Re\left(\left(\int_{t_{k}}^{t_{k}+\tau}e^{s\lambda}\left(h\left(s\right)+a\right)ds\right)\bar{a}\right)\ge\frac{\delta\tau}{2},\]
where $t_{k}=2k\pi/\Im\lambda$. Thus we must have that $\left|a\right|\le C$.\end{proof}
\begin{prop}
\label{pro:normality for asymptotic}Suppose that $A$ is normal,
nonresonant and $n_{0}=2$. Then $S_{A}^{a}\left(B^{n}\right)$ is
a normal family. Furthermore, if $f\in S_{A}^{a}\left(B^{n}\right)$
has $h\in\mc H_{A}\left(B^{n}\right)$ as an infinitesimal generator
(see \prettyref{pro:asymptotic spirallikeness}) then $f$ can be
embedded as the first element of a bounded Loewner chain with infinitesimal
generator $h$.\end{prop}
\begin{proof}
If $U$ is a unitary matrix, $f\in S_{A}^{a}\left(B^{n}\right)$ and
$h\in\mc H_{A}\left(B^{n}\right)$ is an infinitesimal generator for
$f$ then it is straightforward to check that $U^{*}fU\in S_{U^{*}AU}^{a}\left(B^{n}\right)$
and that $U^{*}hU\in\mc H_{U^{*}AU}\left(B^{n}\right)$ is an infinitesimal
generator for $U^{*}fU$. This allows us to assume without loss of
generality that $A=\mr{diag}\left(\lambda_{1},\ldots,\lambda_{n}\right)$
and $\Re\lambda_{1}\ge\ldots\ge\Re\lambda_{n}>0$ (note that $m\left(A\right)=\Re\lambda_{n}$). 

Let $f$ be an $A$-asymptotically spirallike mapping and $h\in\mc H_{A}\left(B^{n}\right)$
be an infinitesimal generator for $f$. Let $v$ be the solution of
\prettyref{eq:transition teq0}. Also, assume that $f$, $h\left(\cdot,t\right)$
and $v\left(\cdot,t\right)$ have the following Taylor series expansions:\begin{eqnarray*}
f\left(z\right) & = & z+F_{2}\left(z^{2}\right)+\ldots\\
h\left(z,t\right) & = & Az+H_{2}\left(z^{2},t\right)+\ldots\\
v\left(z,t\right) & = & e^{-tA}z+V_{2}\left(z^{2},t\right)+\ldots.\end{eqnarray*}
$ $From \prettyref{eq:transition teq0} and then \prettyref{eq:A-A}
one easily gets that\begin{eqnarray*}
e^{tA}V_{2}\left(z^{2},t\right) & = & -\int_{0}^{t}e^{sA}H_{2}\left(\left(e^{-sA}z\right)^{2},s\right)ds\\
 & = & -\int_{0}^{t}e^{-sB_{2}}H_{2}\left(z^{2},s\right)ds.\end{eqnarray*}

As a consequence of \prettyref{pro:asymptotic spirallikeness} , the
above equality and \prettyref{eq:A-A} we have\begin{multline}
F_{2}\left(z^{2}\right)=\lim_{t\rightarrow\infty}\left(e^{tA}V_{2}\left(z^{2},t\right)+e^{tA}F_{2}\left(\left(e^{-tA}z\right)^{2}\right)\right)\\
=\lim_{t\rightarrow\infty}\left(-\int_{0}^{t}e^{-sB_{2}}H_{2}\left(z^{2},s\right)ds+e^{-tB_{2}}F_{2}\left(z^{2}\right)\right).\label{eq:F2}\end{multline}

We want to show that $F_{2}^{\le}$ can be bounded independently of
$f\in S_{A}^{a}\left(B^{n}\right)$. For this we will show that each
of the coefficients $f_{ij}^{k}$ of the monomials $z_{i}z_{j}e_{k}$
from $F_{2}^{\le}$ can be bounded independently of $f\in S_{A}^{a}\left(B^{n}\right)$. 

We know that \[
B_{2}\left(z_{i}z_{j}e_{k}\right)=(\lambda_{i}+\lambda_{j}-\lambda_{k})z_{i}z_{j}e_{k}\]
and so\[
e^{tB_{2}}\left(z_{i}z_{j}e_{k}\right)=e^{t\left(\lambda_{i}+\lambda_{j}-\lambda_{k}\right)}z_{i}z_{j}e_{k}.\]
Projecting \prettyref{eq:F2} on the subspace generated by $z_{i}z_{j}e_{k}$
we get\begin{eqnarray*}
f_{ij}^{k} & = & \lim_{t\rightarrow\infty}\left(-\int_{0}^{t}e^{-s\left(\lambda_{i}+\lambda_{j}-\lambda_{k}\right)}h_{ij}^{k}\left(s\right)ds+e^{-t\left(\lambda_{i}+\lambda_{j}-\lambda_{k}\right)}f_{ij}^{k}\right)\\
 & = & \lim_{t\rightarrow\infty}\left(-\int_{0}^{t}e^{-s\left(\lambda_{i}+\lambda_{j}-\lambda_{k}\right)}\left(h_{ij}^{k}\left(s\right)+\left(\lambda_{i}+\lambda_{j}-\lambda_{k}\right)f_{ij}^{k}\right)ds+f_{ij}^{k}\right)\end{eqnarray*}
($h_{ij}^{k}$$\left(s\right)$ are the coefficients of the monomials
$z_{i}z_{j}e_{k}$ from $H_{2}\left(z^{2},s\right)$). Hence\begin{equation}
\lim_{t\rightarrow\infty}\int_{0}^{t}e^{-s\left(\lambda_{i}+\lambda_{j}-\lambda_{k}\right)}\left(h_{ij}^{k}\left(s\right)+\left(\lambda_{i}+\lambda_{j}-\lambda_{k}\right)f_{ij}^{k}\right)ds=0.\label{eq:necessary condition}\end{equation}
For the coefficients of the monomials of $F_{2}^{\le}$ we have that
$\Re\left(\lambda_{i}+\lambda_{j}-\lambda_{k}\right)\le0$, hence
we can use \prettyref{lem:bounded ass coefficients} and the fact
that $\lambda_{i}+\lambda_{j}-\lambda_{k}\neq0$ $ $(since $A$ is
nonresonant) to conclude that the coefficients of $F_{2}^{\le}$ are
bounded independently of $f$ ($h_{ij}^{k}$ are bounded because $\mc N_{A}$
is compact). $ $

Let $f\left(z,t\right)$ denote the polynomially bounded Loewner chain
with infinitesimal generator $h$ and such that $F_{2}^{\le}\left(z^{2},0\right)=F_{2}^{\le}\left(z^{2}\right)$
(see \prettyref{thm:existence of solutions}). We will see that $f=f\left(\cdot,0\right)$.
This will show that $S_{A}^{a}\left(B^{n}\right)\subset S_{A}^{\mc F}\left(B^{n}\right)$
where \[
\mc F=\left\{ F_{2}^{\le}:f\left(z\right)=z+F_{2}\left(z^{2}\right)+\ldots\in S_{A}^{a}\left(B^{n}\right)\right\} .\]
 By \prettyref{thm:compact class} this yields the normality of $S_{A}^{a}\left(B^{n}\right)$.

It is enough to check that\begin{equation}
0=f\left(z\right)-f\left(z,0\right)=\lim_{t\rightarrow\infty}e^{tA}\left(F_{2}\left(v\left(z,t\right)^{2}\right)-F_{2}\left(v\left(z,t\right)^{2},t\right)\right).\label{eq:necessary for embeddable ass}\end{equation}
We know that (see \prettyref{eq:coefficients}, \prettyref{eq:bounded solution},
\prettyref{eq:polybounded green})\begin{multline*}
F_{2}\left(z^{2},t\right)=e^{tB_{2}}F_{2}^{\le}\left(z^{2}\right)+\int_{0}^{\infty}G_{B_{2}}\left(t-s\right)N_{2}\left(z^{2},s\right)ds\\
=e^{tB_{2}}F_{2}^{\le}\left(z^{2}\right)+\int_{0}^{t}e^{\left(t-s\right)B_{2}}H_{2}^{\le}\left(z^{2},s\right)ds-\int_{t}^{\infty}e^{\left(t-s\right)B_{2}}H_{2}^{+}\left(z^{2},s\right)ds\\
=F_{2}^{\le}\left(z^{2}\right)+e^{tB_{2}}\int_{0}^{t}e^{-sB_{2}}\left(H_{2}^{\le}\left(z^{2},s\right)+B_{2}F_{2}^{\le}\left(z^{2}\right)\right)ds\\
-\int_{0}^{\infty}e^{-sB_{2}}H_{2}^{+}\left(z^{2},s+t\right)ds.\end{multline*}
Substituting the above in \prettyref{eq:necessary for embeddable ass}
we need to verify that\begin{equation}
\lim_{t\rightarrow\infty}e^{tA}e^{tB_{2}}\int_{0}^{t}e^{-sB_{2}}\left(H_{2}^{\le}\left(v\left(z,t\right)^{2},s\right)+B_{2}F_{2}^{\le}\left(v\left(z,t\right)^{2}\right)\right)ds=0\label{eq:first necessary}\end{equation}
and \begin{equation}
\lim_{t\rightarrow\infty}e^{tA}\int_{0}^{\infty}e^{-sB_{2}}\left(H_{2}^{+}\left(v\left(z,t\right)^{2},s\right)-H_{2}^{+}\left(v\left(z,t\right)^{2},s+t\right)\right)ds=0.\label{eq:second necessary}\end{equation}

Using \prettyref{eq:AB}, \prettyref{eq:first necessary} becomes\begin{equation}
\lim_{t\rightarrow\infty}\int_{0}^{t}e^{-sB_{2}}\left(H_{2}^{\le}\left(\left(e^{tA}v\left(z,t\right)\right)^{2},s\right)+B_{2}F_{2}^{\le}\left(\left(e^{tA}v\left(z,t\right)\right)^{2}\right)\right)ds=0.\label{eq:first necessary 2}\end{equation}
Let $v=\left(v_{1},\ldots,v_{n}\right)$. Separating the monomials
in \prettyref{eq:first necessary 2} it is enough to prove that \[
\lim_{t\rightarrow\infty}e^{t\lambda_{i}}v_{i}\left(z,t\right)e^{t\lambda_{j}}v_{j}\left(z,t\right)\int_{0}^{t}e^{-s\left(\lambda_{i}+\lambda_{j}-\lambda_{k}\right)}c_{ij}^{k}\left(s\right)ds=0\]
($c_{ij}^{k}\left(s\right)$ are the coefficients of the monomials
in the polynomial $H_{2}^{\le}\left(\cdot,s\right)+B_{2}F_{2}^{\le}$)
provided that $\Re\left(\lambda_{i}+\lambda_{j}-\lambda_{k}\right)\le0$
and (because of \prettyref{eq:necessary condition})\[
\lim_{t\rightarrow\infty}\int_{0}^{t}e^{-s\left(\lambda_{i}+\lambda_{j}-\lambda_{k}\right)}c_{ij}^{k}\left(s\right)ds=0.\]

It is now enough to check that $e^{t\lambda_{i}}v_{i}\left(z,t\right)$
and $e^{t\lambda_{j}}v_{j}\left(z,t\right)$ are bounded on $\{z\}\times[0,\infty)$,
assuming that $\Re\left(\lambda_{i}+\lambda_{j}-\lambda_{k}\right)\le0$.
This follows from \prettyref{lem:v component bound} provided that
$\Re\lambda_{i},\Re\lambda_{j}<2\Re\lambda_{n}$. Assume that this
is not the case, so for example $\Re\lambda_{i}\ge2\Re\lambda_{n}$.
This implies that \[
\Re\left(\lambda_{i}+\lambda_{j}-\lambda_{k}\right)\ge\Re\left(3\lambda_{n}-\lambda_{1}\right)>0\]
which contradicts $\Re\left(\lambda_{i}+\lambda_{j}-\lambda_{k}\right)\le0$.
For the last inequality we used the hypothesis $n_{0}=2$ which implies
that $2\le\Re\lambda_{1}/\Re\lambda_{n}<3$. 

Separating the monomials in \prettyref{eq:second necessary} it is
enough to prove that\[
\lim_{t\rightarrow\infty}e^{t\lambda_{k}}v_{i}\left(z,t\right)v_{j}\left(z,t\right)\int_{0}^{\infty}e^{-s\left(\lambda_{i}+\lambda_{j}-\lambda_{k}\right)}d_{ij}^{k}\left(s,t\right)ds=0\]
provided that $\Re\left(\lambda_{i}+\lambda_{j}-\lambda_{k}\right)>0$
($d_{ij}^{k}\left(s,t\right)$ are the coefficients of the monomials
in the polynomial $H_{2}^{+}\left(\cdot,s\right)-H_{2}^{+}\left(\cdot,s+t\right)$).
Since $H_{2}^{+}\left(\cdot,s\right)-H_{2}^{+}\left(\cdot,s+t\right)$
can be bounded independently of $s$ and $t$ we see that $\int_{0}^{\infty}e^{-s\left(\lambda_{i}+\lambda_{j}-\lambda_{k}\right)}d_{ij}^{k}\left(s,t\right)ds$
can be bounded independently of $t$. Hence it is enough to check
that \[
\lim_{t\rightarrow\infty}e^{t\lambda_{k}}v_{i}\left(z,t\right)v_{j}\left(z,t\right)=0.\]

If $\Re\lambda_{i},\Re\lambda_{j}\le2\Re\lambda_{n}$ then using \prettyref{lem:v component bound}
we have \[
\left\Vert e^{t\lambda_{k}}v_{i}\left(z,t\right)v_{j}\left(z,t\right)\right\Vert \le C\left(1+t\right)^{2}e^{-t\Re\left(\lambda_{i}+\lambda_{j}-\lambda_{k}\right)}.\]
 If $\Re\lambda_{i}>2\Re\lambda_{n}$ or $\Re\lambda_{j}>2\Re\lambda_{n}$
then using \prettyref{lem:v component bound} again we get\[
\left\Vert e^{t\lambda_{k}}v_{i}\left(z,t\right)v_{j}\left(z,t\right)\right\Vert \le Ce^{-t\Re\left(3\lambda_{n}-\lambda_{k}\right)}.\]
Since $ $ $\Re\left(\lambda_{i}+\lambda_{j}-\lambda_{k}\right)>0$
and $\Re\left(3\lambda_{n}-\lambda_{k}\right)\ge\Re\left(3\lambda_{n}-\lambda_{1}\right)>0$
the above inequalities prove the desired limit. This completes the
proof.\end{proof}
\begin{rem}
It is not clear whether $S_{A}^{a}\left(B^{n}\right)$ is closed under
the assumptions of the above Theorem. Suppose that $\left\{ f_{k}\right\} $
is a sequence in $S_{A}^{a}\left(B^{n}\right)$ converging to some
$f\in S\left(B^{n}\right)$. Let $\left\{ f_{k}\left(z,t\right)\right\} $
be polynomially bounded Loewner chains such that $f_{k}\left(z,0\right)=f_{k}\left(z\right)$.
Then by \prettyref{lem:Loewner subsequence} we have that up to a
subsequence $\left\{ f_{k}\left(z,t\right)\right\} $ converges to
a polynomially bounded Loewner chain $f\left(z,t\right)$ such that
$f\left(z,0\right)=f\left(z\right)$. In order to conclude that $f\in S_{A}^{a}\left(B^{n}\right)$
it would be natural to have that $f$ satisfies \prettyref{eq:as-spir-characterization}
with $v$ satisfying $f\left(v\left(z,t\right),t\right)=f\left(z\right)$.
Unfortunately one can find examples when this doesn't happen.

\begin{rem}\label{rem:Loewner and asymptotic} \prettyref{eq:necessary condition}
gives a necessary condition for a mapping $h\in\mc H_{A}$ to be the
infinitesimal generator associated to some $f\in S_{A}^{a}\left(B^{n}\right)$.
It is possible to choose $h$ such that \prettyref{eq:necessary condition}
is not satisfied for any $f\in S_{A}^{a}\left(B^{n}\right)$. This
means that unlike the $n_{0}=1$ case, there exist polynomially bounded
Loewner chains for which the first element is not from $S_{A}^{a}\left(B^{n}\right)$. 

\end{rem}\end{rem}
\begin{acknowledgement*}
This paper is part of the author's Ph.D. thesis at the University
of Toronto under the supervision of Ian Graham.
\end{acknowledgement*}

\bibliographystyle{plain}
\bibliography{GFT}

\end{document}